\documentclass{jT}

\usepackage{amssymb}
\usepackage{amsthm}
\usepackage{amsmath}
\usepackage{graphicx}
\usepackage{pstricks}

\usepackage[english]{babel}

\newtheorem{thm}{Theorem}[section]
\newtheorem*{thm*}{Theorem}
\newtheorem{prop}[thm]{Proposition}
\newtheorem{lem}[thm]{Lemma}

\newtheorem{cor}[thm]{Corollary}
\newtheorem*{cor*}{Corollary}

\theoremstyle{definition}
\newtheorem{defn}[thm]{Definition}
\newtheorem{alg}[thm]{Algorithm}

\theoremstyle{remark}
\newtheorem{rmk}[thm]{Remark}
\newtheorem{exm}[thm]{Example}

\numberwithin{equation}{section}

\newcommand{\C}{\mathbb C}

\newcommand{\HH}{\mathbb H}

\newcommand{\Q}{\mathbb Q}
\newcommand{\R}{\mathbb R}
\newcommand{\Z}{\mathbb Z}
\newcommand{\Qbar}{\overline{\mathbb Q}}

\newcommand{\frakd}{\mathfrak{d}}
\newcommand{\frakD}{\mathfrak{D}}
\newcommand{\frakH}{\mathfrak{H}}
\newcommand{\frakp}{\mathfrak{p}}

\newcommand{\calO}{\mathcal{O}}

\newenvironment{enumalg}
{\begin{enumerate}}
{\end{enumerate}}

\newenvironment{enumalgalph}
{\begin{enumerate}}
{\end{enumerate}}

\newenvironment{enumroman}
{\begin{enumerate}}
{\end{enumerate}}

\newenvironment{enumalph}
{\begin{enumerate}}
{\end{enumerate}}

\DeclareMathOperator{\ctr}{ctr}

\DeclareMathOperator{\ext}{ext}
\DeclareMathOperator{\impart}{Im}
\DeclareMathOperator{\init}{in}
\DeclareMathOperator{\inter}{int}
\DeclareMathOperator{\N}{N}
\DeclareMathOperator{\nrd}{nrd}
\DeclareMathOperator{\PSL}{PSL}
\DeclareMathOperator{\PSU}{PSU}
\DeclareMathOperator{\rad}{rad}
\DeclareMathOperator{\invrad}{invrad}

\DeclareMathOperator{\reduc}{red}

\DeclareMathOperator{\sgn}{sgn}
\DeclareMathOperator{\SL}{SL}
\DeclareMathOperator{\SU}{SU}
\DeclareMathOperator{\Tr}{Tr}
\DeclareMathOperator{\trd}{trd}

\newcommand{\la}{\langle}
\newcommand{\ra}{\rangle}

\newcommand{\quat}[2]{\displaystyle{\biggl(\frac{#1}{#2}\biggr)}}

\begin{document}

\title[Computing fundamental domains]{Computing fundamental domains \\ for Fuchsian groups}

\author[John {\sc Voight}]{{\sc John} Voight}
\address{John {\sc Voight}\\
Department of Mathematics and Statistics\\ 
16 Colchester Avenue\\ 
University of Vermont\\ 
Burlington, Vermont 05401-1455 \\ 
USA}
\email{jvoight@gmail.com}
\urladdr{http://www.cems.uvm.edu/\~{}voight/}



\maketitle

\begin{resume}
Nous exposons un algorithme pour calculer un domaine de Dirichlet pour un Fuchsian groupe $\Gamma$ avec aire cofinis. En cons\'equence, nous calculons les invariants de $\Gamma$ et une pr\'esentation explicite finis pour $\Gamma$.
\end{resume}

\begin{abstr}
We exhibit an algorithm to compute a Dirichlet domain for a Fuchsian group $\Gamma$ with cofinite area.  As a consequence, we compute the invariants of $\Gamma$, including an explicit finite presentation for $\Gamma$. 
\end{abstr}

\bigskip

Let $\Gamma \subset \PSL_2(\R)$ be a \emph{Fuchsian group}, a discrete group of orientation-preserving isometries of the upper half-plane $\frakH$ with hyperbolic metric $d$.  A \emph{fundamental domain} for $\Gamma$ is a closed domain $D \subset \frakH$ such that:
\begin{enumroman}
\item $\Gamma D = \frakH$, and 
\item $g D^o \cap D^o = \emptyset$ for all $g \in \Gamma \setminus \{1\}$, where ${}^o$ denotes the interior.  
\end{enumroman}
Assume further that $\Gamma$ has cofinite area, i.e., the coset space $X=\Gamma \backslash \frakH$ has finite hyperbolic area $\mu(X)<\infty$; then it follows that $\Gamma$ is finitely generated.

In this article, we exhibit an algorithm to compute a fundamental domain for $\Gamma$; we assume that $\Gamma$ is specified by a finite set of generators $G \subset \SL_2(K)$ with $K \hookrightarrow \R \cap \overline{\Q}$ a number field, and we call $\Gamma$ \emph{exact}.  Suppose that $p \in \frakH$ has trivial stabilizer $\Gamma_p=\{1\}$.  Then the set
\[ D(p)=\{z \in \frakH : d(z,p) \leq d(g z,p) \text{ for all $g \in \Gamma$}\}, \]
known as a \emph{Dirichlet domain}, is a hyperbolically convex fundamental domain for $\Gamma$.  The boundary of $D(p)$ consists of finitely many geodesic segments or \emph{sides}.  We specify $D(p)$ by a sequence of vertices, oriented counterclockwise around $p$.  The domain $D(p)$ has a natural \emph{side pairing}: For each side $s$ of $D(p)$, there exists a unique side $s^*$ and $g \in \Gamma \setminus \{1\}$ such that $s^*=g s$, and the set of such $g$ comprises a set of generators for $\Gamma$.

Our main theorem is as follows.

\begin{thm*}
There exists an algorithm which, given an exact Fuchsian group $\Gamma$ with cofinite area and a point $p \in \frakH$ with $\Gamma_p=\{1\}$, returns the Dirichlet domain $D(p)$, a side pairing for $D(p)$, and a finite presentation for $\Gamma$ with a minimal set of generators.
\end{thm*}

This algorithm also provides a solution to the word problem for the computed presentation of $\Gamma$.  

Of particular and relevant interest is the class of \emph{arithmetic Fuchsian groups}, those groups commensurable with the group of units $\calO_1^*$ of reduced norm $1$ in a maximal order $\calO$ of a quaternion algebra $B$ defined over a totally real field and split at exactly one real place.  Alsina-Bayer \cite{AB} and Kohel-Verrill \cite{KV} give several examples of fundamental domains for arithmetic Fuchsian groups with $F=\Q$.  Our work generalizes that of Johansson \cite{Johansson}, who first made use of a Dirichlet domain for algorithmic purposes: he restricts to the case of arithmetic Fuchsian groups, and we improve on his methods in several respects (see the discussion preceding Algorithm \ref{algdomG} and the reduction algorithms in \S 4).

The algorithm described in the above theorem has the following applications.  The first is a noncommutative generalization of the problem of computing generators for the unit group of a number field.

\begin{cor*}
There exists an algorithm which, given an order $\calO \subset B$ of a quaternion algebra $B$ defined over a totally real field and split at exactly one real place, returns a finite presentation for $\calO_1^*$ with a minimal set of generators.
\end{cor*}

We may also use the presentation for $\Gamma$ to compute invariants.  The group $\Gamma$ has finitely many orbits with nontrivial stabilizer, known as \emph{elliptic cycles} or \emph{parabolic cycles} according as the stabilizer is finite or infinite.  The coset space $X=\Gamma \setminus \frakH$ can be given the structure of a Riemann surface, and we say that $\Gamma$ has \emph{signature} $(g;m_1,\dots,m_t;s)$ if $X$ has genus $g$ and $\Gamma$ has exactly $t$ elliptic cycles of orders $m_1,\dots,m_t \in \Z_{\geq 2}$ and $s$ parabolic cycles.

\begin{cor*}
There exists an algorithm which, given $\Gamma$, returns the signature of $\Gamma$ and a set of representatives for the elliptic and parabolic cycles in $\Gamma$.
\end{cor*}

Finally, we mention a corollary which is useful for the evaluation of automorphic forms. 

\begin{cor*}
There exists an algorithm which, given $\Gamma$ and $z,p \in \frakH$ with $\Gamma_p=\{1\}$, returns a point $z' \in D(p)$ and $g \in \Gamma$ such that $z'=g(z)$.
\end{cor*}

The article is organized as follows.  We begin by fixing notation and discussing the necessary background from the theory of Fuchsian groups (\S 1--2).  We then treat arithmetic Fuchsian groups and give methods for enumerating ``small'' elements of the group $\calO_1^*$, with $\calO \subset B$ a quaternion order as above (\S 3).  Next, we describe the basic algorithm to reduce an element $g \in \Gamma$ with respect to a finite set $G \subset \Gamma$ (\S 4).  We then prove the main theorem (\S 5) and conclude by giving two examples (\S 6).

The author would like to thank the Magma group at the University of Sydney for their hospitality, Steve Donnelly and David Kohel for their helpful input, and Stefan Lemurell for his careful reading of the paper.

\section{Fuchsian groups}

In this section, we present the relevant background from the theory of Fuchsian groups; suggested references include Katok \cite[Chapters 3--4]{Katok} and Beardon \cite[Chapter 9]{Beardon}.  Throughout, we let $\Gamma \subset \PSL_2(\R)$ denote a  Fuchsian group with cofinite area, which is finitely generated by a result of Siegel \cite[Theorem 4.1.1]{Katok}, \cite[\S 1]{Gelfand}.  To simplify, we will identify a matrix $g \in \SL_2(\R)$ with its image in $\PSL_2(\R)$.  

Throughout this section, let $p \in \frakH$ be a point with trivial stabilizer $\Gamma_p=\{1\}$.  Almost all points $p$ satisfy this property: there exist only finitely many $p$ with $\Gamma_p \neq \{1\}$ in any compact subdomain of $\frakH$, and in particular, the set of $p \in \frakH$ with $\Gamma_p \neq \{1\}$ have area zero.  In practice, with probability $1$ a ``random'' choice of $p$ will suffice.

We define the \emph{Dirichlet domain} centered at $p$ to be
\[ D(p)=\{z \in \frakH:d(z,p) \leq d(g z,p) \text{ for all $g \in \Gamma$}\}. \]
The set $D(p)$ is a fundamental domain for $\Gamma$, and is a \emph{hyperbolic polygon}.  More generally, we define a \emph{generalized hyperbolic polygon} to be a closed, connected, and hyperbolically convex domain whose boundary consists of finitely many geodesic segments, called \emph{sides}, so that a hyperbolic polygon is a generalized hyperbolic polygon with finite area.

Let $D \subset \frakH$ be a hyperbolic polygon.  Let $S=S(D)$ denote the set of sides of $D$, with the following convention: if $g \in \Gamma$ is an element of order $2$ which fixes a side $s$ of $D$, and $s$ contains the fixed point of $g$, we instead consider $s$ to be the union of two sides meeting at the fixed point of $g$.  We define a labeled equivalence relation on $S$ by
\[ P=\{(g,s,s^*) : s^*=g(s)\} \subset \Gamma \times (S \times S). \]
We say that $P$ is a \emph{side pairing} for $D$ if $P$ induces a partition of $S$ into pairs, and we denote by $G(P)$ the projection of $P$ to $\Gamma$.  

\begin{prop} \label{sidepair}
The Dirichlet domain $D(p)$ has a side pairing $P$, and the set $G(P)$ generates $\Gamma$.  Conversely, let $D \subset \frakH$ be a hyperbolic polygon, and let $P$ be a side pairing for $D$.  Then $D$ is a fundamental domain for the group generated by $G(P)$.
\end{prop}

\begin{proof}
The first statement is well-known \cite[Theorem 9.3.3]{Beardon}, \cite[Theorem 3.5.4]{Katok}.  For the second statement, we refer to Beardon \cite[Theorem 9.8.4]{Beardon} and the accompanying exercises: the condition that $\mu(D)<\infty$ ensures that any vertex which lies on the circle at infinity is fixed by a hyperbolic element \cite[\S 1]{Gelfand}.
\end{proof}

\begin{rmk}
The second statement of Proposition \ref{sidepair} extends to a larger class of polygons (see \cite[\S 9.8]{Beardon}), and therefore conceivably our results extend to the class of finitely generated non-elementary Fuchsian groups of the first kind.  For simplicity, we restrict to the case of groups with cofinite area.
\end{rmk}

We can define an analogous equivalence relation on the set of vertices of $D$, and we say that a vertex $v$ of $D$ is \emph{paired} if each side $s$ containing $v$ is paired to a side $s^*$ via an element $g \in G$ such that $gv$ is a vertex of $D$.

We now consider the corresponding notions in the hyperbolic unit disc $\frakD$, which will prove more convenient for algorithmic purposes.  The maps
\begin{equation} \label{phimap}
\setlength{\arraycolsep}{0.5ex}
\begin{array}{rlcrl}
\phi:\frakH &\to \frakD & \quad & \phi^{-1}:\frakD &\to \frakH \\
z &\mapsto \displaystyle{\frac{z-p}{z-\overline{p}}} & & w &\mapsto \displaystyle{\frac{\overline{p}w-p}{w-1}}
\end{array}
\end{equation}
define a conformal equivalence between $\frakH$ and $\frakD$ with $p \mapsto \phi(p)=0$.  Via the map $\phi$, the group $\Gamma$ acts on $\frakD$ as
\[ \Gamma^{\phi}=\phi\Gamma\phi^{-1} \subset \PSU(1,1)=\left\{\pm \begin{pmatrix} a & b \\ c & d \end{pmatrix} \in \PSL_2(\C) : a=\overline{d}, b=\overline{c}\right\}. \]
We may analogously define a Dirichlet domain $D(q)$ for $q \in \frakD$ with $\Gamma_q=\{1\}$, and we have $\phi(D(p))=D(0) \subset \frakD$.  To ease notation, we identify $\Gamma$ with $\Gamma^{\phi}$ by $g \mapsto g^{\phi}=\phi g\phi^{-1}$ when no confusion can result.

Any matrix $g=\begin{pmatrix} a & b \\ c & d \end{pmatrix} \in \SU(1,1)$ acts on $\frakD$, multiplying lengths by $|g'(z)|=|cz+d|^{-2}$, and therefore Euclidean lengths (and areas) are preserved if and only if $|cz+d|=1$.  We define the \emph{isometric circle} of $g$ to be
\[ I(g)=\{z \in \C: |cz+d|=1\}; \]
if $c \neq 0$, then $I(g)$ is a circle with radius $1/|c|$ and center $-d/c$, and if $c=0$ then $I(g)=\C$.  We denote by 
\[ \inter(I(g))=\{z \in \C: |cz+d|<1\}, \quad 
\ext(I(g))=\{z \in \C: |cz+d|>1\} \]
the \emph{interior} and \emph{exterior} of $I(g)$, respectively.

With these notations, we now find the following alternative description of the Dirichlet domain $D(0) \subset \frakD$.

\begin{prop} \label{forddom}
\ 
\begin{enumalph}
\item The domain $D(0)$ is the closure in $\frakD$ of 
\[ \bigcap_{g \in \Gamma \setminus \{1\}} \ext(I(g)). \]
\item For any $g \in \SU(1,1)$, we have 
\[ d(z,0) \left.\begin{cases} < \\ = \\ > \end{cases} \hspace{-2.5ex} \right\} d(g z,0) \text{ according as } 
\begin{cases}
z \in \ext(I(g)), \\
z \in I(g),\\
z \in \inter(I(g)).
\end{cases} \]
\end{enumalph}
\end{prop}

\begin{proof}
See Katok \cite[Theorem 3.3.5]{Katok}; we note that if $g \in \Gamma$ and $c=0$, then $q=0$ is a fixed point of $g$, so by hypothesis $g=1$, and hence $\ext(I(g)) \neq \emptyset$ for all $g \neq 1$.  In particular, since $\Gamma$ has cofinite area we note that the intersection in (a) is nonempty.
\end{proof}

\begin{cor} \label{gginv}
For any $g \in \SU(1,1)$, we have $g I(g) = I(g^{-1})$.  
\end{cor}

\begin{proof}
By Proposition \ref{forddom}(b), we have 
\[ w=gz \in I(g^{-1}) \Leftrightarrow d(g^{-1}w,0)=d(w,0) \Leftrightarrow d(z,0)=d(gz,0) \Leftrightarrow z \in I(g) \]
and the result follows.
\end{proof}

\begin{rmk} \label{HtoD}
One can similarly define isometric circles $I(g)$ for $g \in \PSL_2(\R)$ acting on $\frakH$.  One warning is due, however: although $\phi^{-1}(D(0))=D(p) \subset \frakH$ is again a Dirichlet domain, its sides need not be contained in isometric circles (as the map $\phi$ is a hyperbolic isometry, whereas isometric circles are defined by a Euclidean condition).  Instead, we see easily that 
\[ \phi^{-1} I(g^{\phi})=\{z \in \frakH : d(z,p)=d(g z,p)\}, \] 
i.e., the isometric circle $I(g^{\phi})$ corresponds in $\frakH$ to the perpendicular bisector of the geodesic between $p$ and $g(p)$.  In particular, if $p=i$ then a somewhat lengthy calculation reveals that for $g=\begin{pmatrix} a & b \\ c & d \end{pmatrix} \in \SL_2(\R)$, this perpendicular bisector is the half-circle of radius $\displaystyle{\frac{a^2+b^2+c^2+d^2-2}{(a^2+c^2-1)^2}}$ centered at $\displaystyle{\frac{ab+cd}{a^2+c^2-1}} \in \R$.
\end{rmk}

The domain $D(0)$ is also known as a \emph{Ford domain}, since Proposition \ref{forddom} is originally attributed to Ford \cite[Theorem 7, \S 20]{Ford}.  The heart of our algorithm (as provided in the main theorem) will be to algorithmically construct a Ford domain.

\section{Algorithms for the upper half-plane and unit disc}

We represent points $p \in \frakH,\frakD$ using exact complex arithmetic: see Pour-El--Richards \cite{PER}, Weihrauch \cite{Weihrauch} for theoretical foundations (the subject of computable analysis) and e.g.\ Boehm \cite{Boehm}, Gowland-Lester \cite{GL} for a discussion of practical implementations.  Alternatively, our algorithms can be interpreted using fixed and sufficiently large precision; even though one cannot predict in advance the precision required to guarantee correct output, it is likely that an error due to round-off will only very rarely occur in practice; see also Remark \ref{precissue}. The induced action on $\frakD$ has $\Gamma \leftrightarrow \Gamma^{\phi} \subset \SU(1,1)$, represented as matrices with exact complex entries.

A Fuchsian group $\Gamma$ is \emph{exact} if it has a finite set of generators $G \subset \SL_2(K)$ with $K \hookrightarrow \Qbar \cap \R$ a number field; from now on, we assume that the group $\Gamma$ is exact.  Even up to conjugation in $\PSL_2(\R)$, not every finitely generated Fuchsian group is exact; our methods conceivably extend to the case where the set of generators $G \subset \SL_2(\R)$ are specified with (exact) real entries, but we will not discuss this case any further.  Algorithms for efficiently computing with algebraic number fields are well-known (see e.g.\ Cohen \cite{Cohen}). 

We now discuss some elementary methods for working with generalized hyperbolic polygons in $\frakD$, which are defined analogously as those in $\frakH$.  

Let $\overline{\frakD}=\{z \in \C:|z| \leq 1\}$ denote the closure of $\frakD$ and let $\partial \frakD=\{z \in \C:|z|=1\}$ be the \emph{circle at infinity}.  We represent a geodesic $L$ in $\frakD$ in bits by four pieces of data: 
\begin{itemize}
\item the center $c=\ctr(L) \in \C \cup \{\infty\}$,
\item the radius $r=\rad(L) \in \R \cup \{\infty\}$ of $L$, and
\item the \emph{initial point} $z=\init(L) \in \overline{\frakD}$ and the \emph{terminal point} $w \in \overline{\frakD}$;
\end{itemize}
the inital and terminal points are normalized so that the path along $L$ follows a counterclockwise orientation.  Although this data is redundant, it will be more efficient in practice to store all values rather than, say, to recompute $c$ and $r$ when needed.  

If $L_1,L_2 \subset \frakD$ are geodesics which intersect at a point $v \in \frakD \setminus \{0\}$, then we define $\angle(L_1,L_2)$ to be the counterclockwise-oriented angle at $v$ from the geodesics $L_1$ to $L_2$ for the wedge directed toward the origin, so that in particular we have $\angle(L_2,L_1)=-\angle(L_1,L_2)$.

\begin{exm} \label{angle}
In the following figure, we depict a geodesic and the angle $\angle(L_1,L_2) \approx 3\pi/8$ between geodesics.
\begin{center}
\begin{pspicture}(3,-2.5)(11,2.5)
\psclip{\psframe[linecolor=white](3,-2.5)(6,2.5)} \pscircle(0,0){5} \endpsclip
\psclip{\psframe[linecolor=white](3,-2.5)(6,2.5)} \pscircle(5.5,0){2} \endpsclip

\rput(5.7,0){$c$} \pscircle[fillstyle=solid,fillcolor=black](5.5,0){0.05}
\rput(4.5,0.2){$r$} \psline(3.5,0)(5.5,0)
\rput(2.9,-1){$z=\init(L)$} \pscircle[fillstyle=solid,fillcolor=black](3.78,-1){0.05}
\rput(4,1.58){$w$} \pscircle[fillstyle=solid,fillcolor=black](4,1.31){0.05}

\psclip{\psframe[linecolor=white](8,-2.5)(9.95,2.5)} \pscircle(10.5,-1.5){2.3} \endpsclip
\psclip{\psframe[linecolor=white](8,-2.5)(10,2.5)} \pscircle(10.4,2.1){1.9} \endpsclip
\psclip{\psframe[linecolor=white](8,-2.5)(11,2.5)} \pscircle(5,0){5} \endpsclip
\psarcn{->}(9.7,0.55){0.7}{203}{160}

\rput(8.1,2.3){$L_2$}
\rput(8.1,0.6){$\angle(L_1,L_2)$}
\rput(7.9,-1.75){$L_1$}
\end{pspicture} \\
\textbf{Figure \ref{angle}}: Geodesics and angles
\end{center} 
\end{exm}

We leave it to the reader to show that one can compute using elementary formulae the following quantities: for geodesics $L_1,L_2$, the intersection $L_1 \cap L_2$ and (if nonzero) the angle $\angle(L_1,L_2)$, as well as the area of a hyperbolic polygon.  

\begin{defn}
Let $G \subset \Gamma \setminus \{1\}$.  The \emph{exterior domain} of $G$, denoted $E=\ext(G)$, is the closure in $\overline{\frakD}$ of the set $\bigcap_{g \in G} \ext(I(g)) \cap \frakD$.
\end{defn}

With this definition, Proposition \ref{forddom}(a) becomes simply the statement that $\ext(\Gamma \setminus \{1\})$ is the closure of $D(0)$.

Let $G \subset \Gamma$ be a finite subset and let $E=\ext(G)$ be its exterior domain.  Then $E$ is a generalized hyperbolic polygon whose sides are contained in isometric circles $I(g)$ with $g \in G$.  A \emph{proper vertex} of $E$ is a point of intersection $v \in I(g) \cap I(g')$ between two sides (with $g \neq g' \in G$); a \emph{vertex at infinity} of $E$ is a point of intersection $v \in I(g) \cap \partial \frakD$ between a side and the circle of infinity.  A \emph{vertex} of $E$ is either a proper vertex or a vertex at infinity.

\begin{defn} \label{extGdomU}
Let $E=\ext(G)$ be an exterior domain.  A sequence $U=g_1,\dots,g_n$ is a \emph{normalized boundary} for $E$ if:
\begin{enumroman}
\item $E=\ext(U)$; 
\item $I(g_1),\dots,I(g_n)$ contain the counterclockwise consecutive sides of $D$; and
\item the vertex $v \in E$ with minimal $\arg(v) \in (0,2\pi)$ is either a proper vertex with $v \in I(g_1) \cap I(g_2)$ or a vertex at infinity with $v \in I(g_1)$.
\end{enumroman}
\end{defn}

It is clear that for each exterior domain $E$, there exists a unique normalized boundary $G$ for $E$: in (i) and (ii) we order exactly those $g_i$ for which $I(g_i)$ are sides of $E$ and in (iii) we choose a consistent place to start.  

\begin{exm} \label{normbound}
In the following figure, we exhibit a normalized boundary $G=\{g_1,g_2,g_3,g_4\}$; the vertices $v_1,v_2$ are on the circle at infinity whereas $v_3,v_4,v_5$ are proper.
\begin{center}
\begin{pspicture}(-3.5,-3.5)(3.5,3.5)
\pscircle[fillstyle=solid,fillcolor=lightgray](0,0){3}

\psclip{\pscircle(0,0){3}} \pscircle[fillstyle=solid,fillcolor=white](-4.2,-4.2){5.5} \endpsclip
\psclip{\pscircle(0,0){3}} \pscircle[fillstyle=solid,fillcolor=white](-3.5,1){2} \endpsclip
\psclip{\pscircle(0,0){3}} \pscircle[fillstyle=solid,fillcolor=white](-3.5,4.7){5} \endpsclip
\psclip{\pscircle(0,0){3}} \pscircle[fillstyle=solid,fillcolor=white](2.3,-3.5){3} \endpsclip
\psclip{\pscircle(0,0){3}} \pscircle[fillstyle=solid,fillcolor=white](3,-1){1.5} \endpsclip

\psclip{\pscircle(0,0){3}} \pscircle(-4.2,-4.2){5.5} \endpsclip
\psclip{\pscircle(0,0){3}} \pscircle(-3.5,1){2} \endpsclip
\psclip{\pscircle(0,0){3}} \pscircle(-3.5,4.7){5} \endpsclip
\psclip{\pscircle(0,0){3}} \pscircle(2.3,-3.5){3} \endpsclip
\psclip{\pscircle(0,0){3}} \pscircle(3,-1){1.5} \endpsclip

\psline{->}(-3.5,0)(3.5,0) 
\psline{->}(0,-3.5)(0,3.5)

\rput(1.35,3){$v_2$} \pscircle[fillstyle=solid,fillcolor=black](1.1,2.785){0.05}
\rput(-1.1,0.6){$v_3$} \pscircle[fillstyle=solid,fillcolor=black](-1.1,0.33){0.05}
\rput(0.4,-1.6){$v_4$} \pscircle[fillstyle=solid,fillcolor=black](0.4,-1.2){0.05}
\rput(1.8,-0.85){$v_5$} \pscircle[fillstyle=solid,fillcolor=black](1.57,-0.6){0.05}
\rput(3.25,0.55){$v_1$} \pscircle[fillstyle=solid,fillcolor=black](2.94,0.48){0.05}

\rput(0.4,2.4){$I(g_2)$}
\rput(-0.8,-0.5){$I(g_3)$}
\rput(0.8,-0.5){$I(g_4)$}
\rput(2.5,-2.3){$I(g_1)$}
\rput(1.5,1.2){$\ext(G)$}
\end{pspicture} \\
\textbf{Figure \ref{normbound}}: Normalized boundary of a generalized hyperbolic polygon
\end{center} 
\end{exm}

We now detail an algorithm which computes a normalized boundary for a given exterior domain.  

\begin{alg} \label{algdomG}
Let $G \subset \Gamma$ be a finite subset.  This algorithm returns the normalized boundary $U$ of the exterior domain $E=\ext(G)$.  
\begin{enumalg}
\item Initialize $\theta := 0$, $U := \emptyset$, and $L := [0,1]$.
\item Let
\[ H := \{g \in G : \arg(I(g) \cap L) \geq \theta\}. \]
\begin{enumalgalph}
\item If $H = \emptyset$, let $g \in G$ be such that 
\[ \theta := \arg(\init(I(g))) \in \theta + [0,2\pi) \] 
is minimal.
\item If $H \neq \emptyset$, let $g \in H$ be such that 
\[ \theta := \arg(I(g) \cap L) \in \theta + [0,2\pi) \] 
is minimal; if more than one such $g$ exists, let $g$ be the one that minimizes $\angle(L,I(g))$. 
\end{enumalgalph}
Let $U := U \cup \{g\}$ and let $L := I(g) \cap \overline{\frakD}$.
\item If $U=\{g_1,\dots,g_n\}$ and $g_n=g_1$, return $U := \{g_1,\dots,g_{n-1}\}$.  Otherwise, return to Step 2.
\end{enumalg}
\end{alg}

\begin{proof}[Proof of correctness]
By definition $\ext(U)$ is a generalized hyperbolic polygon.  Suppose that $E \neq \ext(U)$.  Then there exists $g \in G$ such that $L = I(g) \cap \ext(U)$ is not just a vertex of $\ext(U)$.  Consider the initial point $z=\init(L)$: either $z$ lies on a side $I(g_i)$ of $\ext(U)$ or $z \in \partial \frakD$.  

Suppose that $z \in I(g_i)$.  Let $v_i$ be the initial vertex of the side $s_i \subset I(g_i)$.  Then in the $i$th iteration of Step 2 of the algorithm we have $g \in H$, so the terminal vertex $v_{i+1}$ of $s_i$ is proper and we are in case (b).  But by assumption we have $d(v_i,z) \leq d(v_i,v_{i+1})$ since $I(g_i)$ is a geodesic, and $\arg$ increases along $s_i$ with the distance, thus according to the stipulations of the algorithm we must have $z=v_{i+1}$.  But then in order for the interior of $I(g)$ to intersect $\ext(U)$ nontrivially, we must have $\angle(L,I(g_i)) < \angle(I(g_{i+1}),I(g))$, a contradiction.  

\begin{center}
\begin{pspicture}(-3.5,0)(3.5,3.5)
\psarc[fillstyle=solid,fillcolor=lightgray](0,0){3}{0}{180}

\psclip{\psarc(0,0){3}{0}{180}} \pscircle[fillstyle=solid,fillcolor=white](3.2,1){2} \endpsclip
\psclip{\psarc(0,0){3}{0}{180}} \pscircle[fillstyle=solid,fillcolor=white](0,4.5){3} \endpsclip
\psclip{\psarc(0,0){3}{0}{180}} \pscircle[fillstyle=solid,fillcolor=white](-3.2,1.5){2} \endpsclip

\psarc(0,0){3}{0}{180}
\psclip{\psarc(0,0){3}{0}{180}} \pscircle(3.2,1){2} \endpsclip
\psclip{\psarc(0,0){3}{0}{180}} \pscircle(0,4.5){3} \endpsclip
\psclip{\psarc(0,0){3}{0}{180}} \pscircle(-3.2,1.5){2} \endpsclip
\psclip{\psarc(0,0){3}{0}{180}} \pscircle[linestyle=dashed](-2,3.5){2.5} \endpsclip

\psline{->}(-3.5,0)(3.5,0) 

\rput(1.65,1.65){$v_i$} \pscircle[fillstyle=solid,fillcolor=black](1.45,1.9){0.05}
\rput(-1.6,1.6){$v_{i+1}$} \pscircle[fillstyle=solid,fillcolor=black](-1.25,1.8){0.05}
\rput(-0.45,1.8){$z$} \pscircle[fillstyle=solid,fillcolor=black](-0.45,1.55){0.05}

\rput(0.4,1.2){$I(g_i)$}
\rput(-2.1,2.8){$I(g_{i+1})$}
\rput(0.5,3.25){$I(g)$}

\end{pspicture} \\
\end{center} 

So suppose that $z \in \partial \frakD$.  Then there exists $i$ such that $z$ lies on the principal circle between the terminal point of $I(g_i)$ and the initial point of $I(g_{i+1})$.  But then $\arg(\init(I(g))) < \arg(\init(I(g_{i+1}))$, contradicting (a).  This proves that (i) holds in Definition \ref{extGdomU}.  

It is obvious that (ii) holds, and condition (iii) holds by initialization: if the vertex $v \in E$ with minimal $\arg(v) \in (0,2\pi)$ is a vertex at infinity then it is found in the first iteration of the algorithm in stage (a), and if it is a proper vertex then it is found in the second iteration in stage (b).
\end{proof}

A Ford domain $D(0)$ is specified in bits by a normalized boundary $G$ for $D(0)$.  We can similarly specify a Dirichlet domain $D(p)$ by an analogously defined normalized boundary of perpendicular bisectors, as in Remark \ref{HtoD}; for many purposes, it will be sufficient to represent $D(p)$ by a sequence of vertices (ordered in a counterclockwise orientation around $p$).

\begin{rmk} \label{precissue}
Although the intermediate computations as above are of a numerical sort, an algorithm to compute a Dirichlet domain accepts exact input and produces exact output.
\end{rmk}

\section{Element enumeration in arithmetic Fuchsian groups}

In this section, we treat arithmetic Fuchsian groups, and in particular we exhibit methods for enumerating ``small'' elements of these groups.  See Vigneras \cite{Vigneras} for background material and Voight \cite[Chapter 4]{Voight} for a discussion of algorithms for quaternion algebras.

Let $F$ be a number field with $[F:\Q]=n$ and discriminant $d_F$.  A \emph{quaternion algebra} $B$ over $F$ is an $F$-algebra with generators $\alpha,\beta \in B$ such that 
\[ \alpha^2=h, \quad \beta^2=k, \quad \beta\alpha=-\alpha\beta \]
with $h,k \in F^*$; such an algebra is denoted $B=\quat{h,k}{F}$ and is specified in bits by $h,k \in F^*$.  An element $\gamma \in B$ is represented by $\gamma = x+y\alpha+z\beta+w\alpha\beta$ with $x,y,z,w \in F$, and we define the \emph{reduced trace} and \emph{reduced norm} of $\gamma$ by $\trd(\gamma)=2x$ and $\nrd(\gamma)=x^2-hy^2-kz^2+hkw^2$, respectively.  

Let $B$ be a quaternion algebra over $F$ and let $\Z_F$ denote the ring of integers of $F$.  An \emph{order} $\calO \subset B$ is a finitely generated $\Z_F$-submodule with $F\calO=B$ which is also subring; an order is \emph{maximal} if it is not properly contained in any other order.  We represent an order by a \emph{pseudobasis} over $\Z_F$; see Cohen \cite[\S 1]{Cohen2} for methods of computing with finitely generated modules over Dedekind domains using pseudobases.

A place $v$ of $F$ is \emph{split} or \emph{ramified} according as $B_v = B \otimes_F F_v \cong M_2(F_v)$ or not, where $F_v$ denotes the completion at $v$.  The set $S$ of ramified places of $B$ is finite and of even cardinality, and the ideal $\frakd=\prod_{v \in S, v \nmid \infty} \frakp_v$ of $\Z_F$ is called the \emph{discriminant} of $B$.  

Now suppose that $F$ is a totally real field, and there is a unique split real place $v \not\in S$ corresponding to $\iota_\infty:B \hookrightarrow M_2(\R)$.  Let $\calO \subset B$ be an order and let $\calO_1^*$ denote the group of units of reduced norm $1$ in $\calO$.  Then the group $\Gamma(\calO)=\iota_\infty(\calO_1^*/\{\pm 1\}) \subset \PSL_2(\R)$ is a Fuchsian group \cite[\S\S 5.2--5.3]{Katok}.  If $\calO$ is maximal, we denote $\Gamma^B(1)=\Gamma(\calO)$.  An \emph{arithmetic Fuchsian group} $\Gamma$ is a Fuchsian group commensurable with $\Gamma^B(1)$ for some choice of $B$.  One can, for instance, recover the usual modular groups in this way, taking $F=\Q$, $\calO=M_2(\Z) \subset M_2(\Q)=B$, and $\Gamma \subset \PSL_2(\Z)$ a subgroup of finite index.


An arithmetic Fuchsian group $\Gamma$ has cofinite area; indeed, by a formula of Shimizu \cite[Appendix]{Shimizu}, the area $A=\mu(X)=\mu(\Gamma \backslash \frakH)$ is given by
\begin{equation} \label{shimizu}
A = \frac{4}{(2\pi)^{2n}} d_F^{3/2} \zeta_F(2) \Phi(\frakd) [\Gamma^B(1):\Gamma],
\end{equation}
where $\zeta_F(s)$ denotes the Dedekind zeta function of $F$, and
\[ \Phi(\frakd)=\#(\Z_F/\frakd\Z_F)^*=\N(\frakd)\prod_{\frakp \mid \frakd} \left(1-\frac{1}{\N(\frakp)} \right); \]
here the hyperbolic area is normalized so that 
\[ \mu(\Omega)=\frac{1}{2\pi}\int\!\!\int_\Omega \frac{dx\,dy}{y^2} \] 
and hence an ideal triangle has area $1/2$.  

\begin{rmk} \label{areacompute}
The area $A$ is effectively computable from the formula (\ref{shimizu}).  By the Riemann-Hurwitz formula, we have
\begin{equation} \label{RH}
A = 2g-2+\sum_q e_q\left(1-\frac{1}{q}\right) + e_\infty
\end{equation}
where $e_q$ is the number of elliptic cycles of order $q \in \Z_{\geq 2}$ in $\Gamma$ and $e_\infty$ the number of parabolic cycles.  In particular, $A \in \Q$; and since $e_q>0$ implies $F(\zeta_{2q}) \hookrightarrow B$, the denominator of $A$ is bounded by the least common multiple of all $q$ such that $[F(\zeta_{2q}):F]=2$ (which in particular requires that $F$ contains the totally real subfield $\Q(\zeta_{2q})^+$ of $\Q(\zeta_{2q})$).  Therefore, it suffices to compute the usual Dirichlet series or Euler product expansion for $\zeta_F(2)$ with the required precision; see also Dokchitser \cite{Dokchitser}.
\end{rmk}

We now relate isometric circles to the arithmetic of $B$.  Let $p \in \frakH$ have $\Gamma_p=\{1\}$.  A short calculation with the maps defined in (\ref{phimap}) shows that if $g=\begin{pmatrix} a & b \\ c & d \end{pmatrix} \in \SL_2(\R)$, then $g^{\phi}=\phi g\phi^{-1} \in \SU(1,1)$ has radius
\[ \rad(I(g^{\phi}))=\displaystyle{\frac{2\impart(p)}{|f_g(p)|}}, \]
where $f_g(t)=ct^2+(d-a)t-b$, a polynomial whose roots are the fixed points of $g$ in $\C$.  We will abbreviate $\rad(g)=\rad(I(g^{\phi}))$.  The map 
\begin{equation} \label{invrad}
\setlength{\arraycolsep}{0.5ex}
\begin{array}{rl}
\invrad:M_2(\R) &\to \R \\
g &\mapsto |f_g(p)|^2 +2y^2\det(g) 
\end{array}
\end{equation}
yields a quadratic form on $M_2(\R)$: explicitly, if $p=x+yi$, we have
\begin{align*} 
\invrad\!\begin{pmatrix} a & b \\ c & d \end{pmatrix} &=\left(xa + b - (x^2-y^2) c - xd\right)^2 + y^2\left(a-2xc-d\right)^2 + 2y^2(ad-bc) \\
&= y^2(a-xc)^2+(xa+b-x^2c-xd)^2+y^4c^2+y^2(xc+d)^2,
\end{align*}
and hence the form $\invrad$ is positive definite and via $\iota_\infty$ induces a positive definite form $\invrad:B \to \R$.  For $g \in B$, we note that $\det \iota_\infty(g)=v(\nrd(g))$, where $v$ is the unique split real place of $B$.  

Suppose that $p=i$.  Then we have simply $\invrad\!\begin{pmatrix} a & b \\ c & d \end{pmatrix} = a^2+b^2+c^2+d^2$.  Let $B=\quat{h,k}{F}$.  Identify $F$ with its image $F \hookrightarrow \R$ under the unique split real place of $B$; without loss of generality, we may assume that $h>0$. We may therefore embed $\iota_\infty:B \hookrightarrow M_2(\R)$ by letting
\begin{equation} \label{embedmin}
 \alpha \mapsto \begin{pmatrix} \sqrt{h} & 0 \\ 0 & -\sqrt{h} \end{pmatrix}, \quad 
\beta \mapsto \begin{pmatrix} 0 & \sqrt{|k|} \\ \sgn(k)\sqrt{|k|} & 0 \end{pmatrix}
\end{equation}
where $\sgn$ denotes the sign.  Therefore if $g=x+y\alpha+z\beta+w\alpha\beta \in B$, then we see directly that 
\[ \invrad(g)=x^2+hy^2+|k|z^2+h|k|w^2. \]


For the ramified real places $v$ of $F$, corresponding to $B \hookrightarrow B \otimes_F \R \cong \HH$, the reduced norm form $\nrd_v:B \to \R$ by $g \mapsto v(\nrd(g))$ is positive definite.  Putting these together, we find that the \emph{absolute reduced norm}
\begin{align*}
N:B &\to \R \\
g &\mapsto 2y^2 \textstyle{\sum_{v \in S, v \mid \infty}} \nrd_v(g) + \invrad(g) = |f_g(p)|^2 + 2y^2 \Tr_{F/\Q} \nrd(g)
\end{align*}
is positive definite and gives $\calO$ the structure of a lattice of rank $4n$.  

The elements $g \in \calO$ with small absolute reduced norm $N$ are those such that $|f_g(p)|$ and $\Tr_{F/\Q} \nrd(g)$ are both small---in particular, this will include the elements of $\calO_1^*$ with small $\invrad$ (with respect to $p \in \frakH$), which correspond to elements $g \in \Gamma$ whose isometric circle in $\frakD$ (centered at $p$) has large radius.  Since the Dirichlet domain $D(p)$ has only finitely many sides, those $g \in \Gamma$ with $\rad(g)$ sufficiently small radius cannot contribute to the boundary of $D(p)$.  

Hence, one simple idea to construct $D(p)$ would be to enumerate all elements of $\calO_1^*$ by increasing absolute reduced norm $N$ until the exterior domain of these elements has area equal to $\mu(\Gamma\backslash\frakH)$.  This method shows that $D(p)$ is indeed computable, and may have been known to Klein; it is mentioned by Katok \cite{Katok2} when $F=\Q$ and sees further explication by Johansson \cite{Johansson}.  Using the above framework, we can immediately improve upon this method by enumerating such elements efficiently using lattice reduction, as follows.

\begin{alg} \label{enumlll}
Let $\calO \subset B$ be a quaternion order.  This algorithm returns a Dirichlet domain for $\Gamma(\calO)$.

\begin{enumalg}
\item Compute $A=\mu(\Gamma(\calO)\backslash \frakH)$ by Remark \ref{areacompute}.
\item Embed $\calO \hookrightarrow \R^{4n}$ as a lattice using the absolute reduced norm form $N$, and choose $C \in \R_{>0}$.
\item Using the Fincke-Pohst algorithm \cite{FinckePohst}, compute the set
\[ G(C)=\left\{\iota_\infty(g/u): \pm g \in \calO,\ N(g) \leq C,\ \nrd(g)=u^2 \in \Z_F^{*2}\right\} \subset \Gamma. \]
\item From Algorithm \ref{algdomG}, compute $E=\ext(G(C))$.  If $\mu(E)=A<\infty$, then return $E$; otherwise, increase $C$ and return to Step 2.
\end{enumalg}
\end{alg}

\begin{rmk} \label{chooseC}
In choosing $C$, we note that
\[ \{g \in \calO_1^*: \rad(g) \geq R\}=\left\{g \in \calO: N(g) \leq 2y^2\left(n+\frac{2}{R^2}\right)\right\} \cap \calO_1^*; \]
in practice, we would like to take $C$ large enough so that $G(C) \neq \emptyset$ but not too large.  It is not immediately clear how to choose $C$ (and a strategy for its incrementation) optimally in general, unless one knows something about the radii of the sides of the Dirichlet domain.
\end{rmk}

Our final algorithm (Algorithm \ref{algDO}) significantly improves on Algorithm \ref{enumlll} by the use of a reduction algorithm, which we introduce in the next section.

\section{Reduction algorithm}

In this section, we introduce the reduction algorithm (Algorithm \ref{algred}) which forms the heart of the paper.  This algorithm will allow us to find a normalized basis for the group $\Gamma$ (Algorithm \ref{algD}), yielding a fundamental domain.  

Throughout this section, let $G=\{g_1,\dots,g_t\} \subset \Gamma \setminus \{1\}$ be an (ordered) finite subset of a Fuchsian group $\Gamma$, and denote by $\la G \ra$ the group generated by $G$.  For any $z \in \frakD$, we have a map
\begin{align*}
\rho: \Gamma &\to \R_{\geq 0} \\
\gamma &\mapsto \rho(\gamma;z)=d(\gamma z, 0)
\end{align*}
where $d$ denotes hyperbolic distance.  We abbreviate $\rho(\gamma;0)=\rho(\gamma)$.

\begin{defn}
Let $z \in \frakD$.  An element $\gamma \in \Gamma$ is \emph{$(G,z)$-reduced} if for all $g \in G$, we have $\rho(\gamma;z) \leq \rho(g \gamma;z)$, and $\gamma$ is \emph{$G$-reduced} if it is $(G,0)$-reduced.
\end{defn}

\begin{rmk} \label{rmkgred}
By Proposition \ref{forddom}, we note that $\gamma$ is $(G,z)$-reduced if and only if $\gamma z \in \ext(G)$.   
\end{rmk}

We arrive at the following straightfoward algorithm to perform $(G,z)$-reduction.  

\begin{alg} \label{algred}
Let $\gamma \in \Gamma$ and let $z \in \frakD$.  This algorithm returns elements $\overline{\gamma} \in \Gamma$ and $\delta \in \la G \ra$ such that $\overline{\gamma}$ is $(G,z)$-reduced and $\overline{\gamma}=\delta\gamma$.
\begin{enumalg}
\item Initialize $\overline{\gamma} := \gamma$ and $\delta := 1$.
\item If $\rho(\overline{\gamma};z) \leq \rho(g \overline{\gamma};z)$ for all $g \in G$, return $\overline{\gamma}, \delta$.  Otherwise, let $g \in G$ be the first element in $G$ such that 
\[ \rho(g \overline{\gamma};z) = \min_i \rho(g_i \overline{\gamma};z). \]
Let $\overline{\gamma} := g_i \overline{\gamma}$ and $\delta := g_i \delta$, and return to Step $2$.
\end{enumalg}
\end{alg}

We denote the output of the above algorithm $\overline{\gamma}=\reduc_G(\gamma;z)$ and abbreviate $\reduc_G(\gamma;0)=\reduc_G(\gamma)$.

\begin{proof}[Proof of correctness]
The output of the algorithm $\overline{\gamma}$ is by definition $G$-reduced.  The algorithm terminates because if $\overline{\gamma}_1,\overline{\gamma}_2,\dots$ are the elements that arise in the iteration of Step 2, then $\rho(\overline{\gamma}_1;z)>\rho(\overline{\gamma}_2;z)>\dots$; however, the action of $\Gamma$ is discrete, so among the points $\{\overline{\gamma}_i(z)\}_i$, only finitely many are distinct.
\end{proof}

A priori, Step $2$ in Algorithm \ref{algred} depends on the ordering of the set $G$ and therefore the output $\overline{\gamma}$ will depend on this ordering.  This is analogous to the situation of the reduction theory of polynomials, as follows.  Let $k$ be a field, let $R=k[x_1,\dots,x_n]$ be the polynomial ring over $k$ in $n$ variables with a choice of term order, and let $G=g_1,\dots,g_t \in R$ be not all zero.  Applying the generalized division algorithm, one can reduce a polynomial $f \in R$ with respect to $G$, and the result is unique (i.e., independent of the ordering of the $g_i$) for all $f$ if $G$ is a Gr\"obner basis of the ideal $I=\la g_1, \dots, g_t \ra$.  Moreover, if $G$ is a Gr\"obner basis, then $f \in I$ if and only if the remainder on division of $f$ by $G$ is zero.  (See e.g.\ Cox-Little-O'Shea \cite[Chapter 2]{CLO}.)  We can prove analogous statements, replacing the ring $R$ by the group $\Gamma$, as follows.

\begin{prop} \label{normbas}
Suppose that $\ext(G)$ is a fundamental domain for $\la G \ra$.  Then for almost all $z \in \frakD$, $\reduc_G(\gamma;z)$ as an element of $\Gamma$ is independent of the ordering of $G$ for all $\gamma \in \la G \ra$.  Moreover, for all $\gamma \in \Gamma$, we have $\reduc_G (\gamma)=1$ if and only if $\gamma \in \la G \ra$.
\end{prop}

Here, ``almost all'' means for all $z$ outside of a set of measure zero: it suffices to take $z$ in the $\Gamma$-orbit of the interior of $\ext(G)$.  

\begin{proof}
Suppose that $\ext(G)$ is a fundamental domain for $\la G \ra$.  Let $z$ be in the $\Gamma$-orbit of $z_0 \in \inter(\ext(G))$, let $\gamma \in \la G \ra$, and let $\overline{\gamma}=\reduc_G(\gamma;z)$.  Then by Remark \ref{rmkgred}, we have $\overline{\gamma}z \in \ext(G)$, and since $\ext(G)$ is a fundamental domain and $\Gamma z = \Gamma z_0$ with $z_0 \in \inter(\ext(G))$, we must have $\overline{\gamma}z=z_0$; in particular, $\overline{\gamma}$ is unique and independent of the ordering of $G$.  The second statement follows similarly: we have that $0,\overline{\gamma}(0) \in \inter(\ext(G))$, so if $\overline{\gamma} \neq 1$ then $\gamma \not\in \la G \ra$.
\end{proof}

Inspired by the preceding proposition, we make the following definition.

\begin{defn}
A set $G$ is a \emph{basis for $\Gamma$} if $\ext(G)$ is a fundamental domain for $\la G \ra=\Gamma$.  If $G$ is a basis that  forms a normalized boundary for $\Gamma$, then we say that $G$ is a \emph{normalized basis}.
\end{defn}

\begin{rmk} \label{wordprob}
It follows from Proposition \ref{normbas} that if one can compute a normalized basis $G$ for $\Gamma$, then one also has a solution to the word problem: given any element $\gamma \in \Gamma$, we compute $\overline{\gamma}=\reduc_G \gamma$, which by Proposition \ref{normbas} must satisfy $\overline{\gamma}=1$, so we have explicitly written $\gamma$ as a word from $G$.
\end{rmk}

We construct a normalized basis for as follows. 

\begin{alg} \label{algD}
Let $G \subset \Gamma$.  This algorithm returns a normalized basis for $\la G \ra$.
\begin{enumalg}
\item Let $G := \{g_1,\dots,g_t,g_1^{-1},\dots,g_t^{-1}\}$.
\item Compute the normalized boundary $U$ of $\ext(G)$ by Algorithm \ref{algdomG}.
\item Let $G' := U$.  For each $g \in G$, compute $\overline{g}=\reduc_{G \setminus \{g\}}(g)$ using Algorithm $\ref{algred}$. If $\overline{g} \neq 1$, set $G':=G' \cup \{\overline{g}\}$.  
\item Compute the normalized boundary $U'$ of $\ext(G')$.  If $U'=U$, set $G := G'$ and proceed to Step 5; otherwise set $U := U'$ and return to Step 3.
\item If all vertices of $E=\ext(U)$ are paired, return $U$.  Otherwise, for each $g \in G$ with a vertex $v \in I(g)$ which is not paired, compute $\overline{g} := \reduc_G(g;v)$, where if $v$ is a vertex at infinity we replace $v$ by a nearby point in $I(g^{-1}) \setminus E \subset \frakD$.  Add the reductions $\overline{g}$ for each nonpaired vertex $v$ to $G$ and return to Step $2$.
\end{enumalg}
\end{alg}

\begin{proof}[Proof of correctness]
First, note that if $v$ be a vertex of $E=\ext(G)$, then by Corollary \ref{gginv}, $v$ is a paired vertex if and only if for every side $s \subset I(g)$ containing $v$, we have that $gv \in I(g^{-1})$ is a vertex of $E$.

Next, we prove that if the algorithm terminates it does so correctly.  We construct a side pairing as in \S 1.  A side $s$ of $E$ pairs up with $gs \subset I(g^{-1})$ if and only if its vertices are paired, necessarily with the vertices of $I(g^{-1})$ by Corollary \ref{gginv}.  Therefore if we terminate in Step 5, we have in fact paired all sides of $\ext(U)$ and by Theorem \ref{sidepair}, $\ext(U)$ is a Dirichlet domain and $U$ is a basis.  

Otherwise, by Step 5 we have $v \in s$ such that $gv \not \in \ext(G)$.  We now compute $\overline{g}=\reduc_G(g; v)$, and refer to Proposition \ref{forddom}.  Since $v \in I(g)$, we have $d(v,0)=d(gv,0)$, and since $gv \not\in \ext(G)$, we have $d(gv,0) > d(\overline{g} v, 0)$.  Putting these together, we find that $v \in \inter(I(\overline{g}))$ and hence $\ext(G \cup \{\overline{g}\}) \subsetneq \ext(G)$.  

Consider now the limit of the sets $G_\infty=\lim G$ and $U_\infty=\lim U$ as we let the algorithm run forever.  Accordingly, every vertex $v$ of $\ext(U_\infty)$ must be paired, otherwise it would be caught in some step of the algorithm.  Therefore by the above, $U_\infty$ is a basis for $\la G_\infty \ra$.  But at each step of the algorithm, the group $\la G \ra$ remains the same, even as $G$ changes: indeed, in Step 3, if $\overline{g}=1$ then $g \in \la G \setminus \{g\}\ra$.  Therefore $\la G_\infty \ra = \la G \ra$, and since $\la G \ra$ is finitely generated we know that $U$ is finite, and hence the algorithm terminates after finitely many steps.
\end{proof}

We now extend this in the natural way to an arithmetic Fuchsian group $\Gamma(\calO)$.  

\begin{alg} \label{algDO}
Let $\calO$ be a quaternion order.  This algorithm returns a basis $G$ for $\Gamma=\Gamma(\calO)$.
\begin{enumalg}
\item Choose $C \in \R_{>0}$, initialize $G := \emptyset$, and compute $A=\mu(\Gamma \backslash \frakH)$.
\item Using Steps 1--2 in Algorithm \ref{enumlll}, compute the set $G(C) \subset \Gamma$.
\item Call Algorithm \ref{algD} with input $G \cup G(C)$ and let $G$ be the output.  If $\mu(\ext(G))=A<\infty$, then return $G$; otherwise, increase $C$ and return to Step 2.  
\end{enumalg}
\end{alg}

A fundamental domain for an arithmetic Fuchsian group $\Gamma \subset \Gamma(\calO)$ can easily be computed from this by first running Algorithm \ref{algDO} and then computing a coset decomposition of $\Gamma$ in $\Gamma(\calO)$; and for that reason, one may even restrict consideration to the case where $\calO$ is maximal.  

\begin{rmk}
In practice, in some cases we can improve Step 5 of Algorithm \ref{algD} for arithmetic Fuchsian groups as follows.  For each nonpaired vertex $v$, we can consider those elements with small absolute reduced norm $N$ relative to $p \in \frakD$ taken to be a point along the geodesic between $0$ and $v$: indeed, by continuity if $g \in \calO_1^*$ has $v \in \inter(I(g))$, then $\rad(g)$ increases as the center $p$ moves towards $v$ and thus $N(g)$ decreases, so using lattice reduction we are likely to find a small such $g$.
\end{rmk}

\section{Proof of the main theorem}

We are now ready to prove the main theorem of this paper.

\begin{thm}
There exists an algorithm which, given a finitely generated Fuchsian group $\Gamma$ and a point $p \in \frakH$ with $\Gamma_p=\{1\}$, returns the Dirichlet domain $D(p)$, a side pairing for $D(p)$, and a finite presentation for $\Gamma$ with a minimal set of generators.
\end{thm}

To prove the theorem, we need to show how the output of Algorithm \ref{algD} yields a finite presentation for $\Gamma$ with a minimal set of generators.  Indeed, Algorithm \ref{algD} terminates only if it has computed a side pairing $P$ (which we may assume meets the convention in \S 1) for the Dirichlet domain $D$.  Such a side pairing $P$ gives a set $G$ of generators for $\Gamma$ by Proposition \ref{sidepair}.  

We now consider the induced relation on the set of vertices.  A \emph{cycle} of $D$ is a sequence $v_1,\dots,v_n=v_1$ which is the (ordered) intersection of the $\Gamma$-orbit of $v=v_1$ with $D$.  To each cycle, we associate the word $g=g_ng_{n-1}\cdots g_2g_1$ where $g_i(v_i)=v_{i+1}$ and the indices are taken modulo $n$.  We say that a cycle is a \emph{pairing cycle} if $g_i \in G$ for all $i$, and without further mention we shall assume from now on that a cycle is a pairing cycle.

A cycle is \emph{minimal} if $v_i \neq v_j$ for all $i \neq j$.  Every vertex $v$ of $D$ is contained in a unique minimal cycle (up to reversion and cyclic permutation).  Indeed, by the uniqueness of the side pairing, a vertex $v \in I(g) \cap I(g')$ either has $v=gv=g'v$, in which case $v$ has nontrivial stabilizer and one has the singleton cycle $v$, or $v$ has trivial stabilizer and is paired with the distinct elements $gv \in I(g^{-1})$ and $g'v \in I(g'^{-1})$, each of which also has trivial stabilizer, and then continuing in this way one constructs a (unique minimal) cycle.  This analysis gives rise to the following algorithm.

\begin{alg} \label{mincycles}
Let $P$ be a side pairing for a Dirichlet domain $D$ for $\Gamma$.  This algorithm returns a set of minimal cycles for $D$.
\begin{enumalg}
\item Initialize $V$ to be the set of vertices of $D$ and $M := \emptyset$.
\item If $V = \emptyset$, terminate.  Otherwise, choose $v \in V$ with $v=I(g) \cap I(g')$ for $g,g' \in G(P)$.  If $gv=v$, add the cycle $v$ to $M$ and remove $v$ from $V$, and return to Step 2.  Otherwise, let $i := 1$ and $v_1 := v$.
\item Let $v_{i+1} := gv_i \in I(g^{-1}) \cap I(g')$.  If $v_{i+1}=v_1$, add the cycle $v_1,\dots,v_i$ to $M$, remove these elements from $V$, and return to Step 2; otherwise, increment $i := i+1$, let $g := g'$ and return to Step 3.
\end{enumalg}
\end{alg}

The relations associated to minimal cycles have the following important property.

\begin{lem}
Let $g \in G$ be a side-pairing element.  Then $g$ appears at most once in any word associated to a minimal cycle.  Moreover, $g$ and its inverse appears in exactly two such words.
\end{lem}

\begin{proof}
By definition, a side-pairing element $g$ pairs a unique set of sides: in particular, $g$ pairs the vertices of one side $s$ with the vertices of another.  Suppose that $g$ occurs twice in a word associated to a minimal cycle.  Then by minimality, the vertices of $s$ are in the same $\Gamma$ orbit.  But this implies that $g$ maps $I(g)$ to itself, so $g$ has order $2$ and therefore one of the vertices of $s$ is fixed by $g$, a contradiction.

In a similar way, we see that $g$ and its inverse can appear in at most two words since each vertex belongs to exactly one minimal cycle.
\end{proof}

We have the following characterization of the minimal cycles.

\begin{prop}{Beardon \cite[Theorem 9.4.5]{Beardon}} \label{areazero}
For all $p \in \frakH$ outside of a set of area zero, the following statements hold:
\begin{enumroman}
\item Every elliptic cycle has length $1$;
\item Every accidental cycle has length $3$; and
\item Every parabolic cycle has length $1$.
\end{enumroman}
\end{prop}

\begin{rmk} \label{exceptp2}
The exceptional set of $p$ is contained in the union
\[ E_2=\bigcup_{f,g,h \in \Gamma} \{z : R(z) \in \R\} \]
over all triples $f,g,h \in \Gamma$ such that
\[ R(z)=\frac{(z-gz)(fz-hz)}{(z-fz)(gz-hz)} \]
is not constant.  It is easy to see that the set $E_2$ has area zero.  
\end{rmk}

For the purposes of computing a minimal set of generators and relations, we may and do assume that $p$ does not lie in the exceptional set; indeed, a sufficiently general choice of $p$ will suffice, and so in practice the conditions of Proposition \ref{areazero} always hold.  In particular, every elliptic cycle is represented by a minimal cycle (whose fixed point is a vertex of $D$).

Now, to each cycle, associated to the word $g$, we further associate a relation in $\Gamma$ as follows.  By definition, we have $g \in \Gamma_v$, and therefore we have one of three possibilities.  If $\#\Gamma_v=1$, then we have the relation $g=1$; we call $g$ an \emph{accidental cycle}.  If $1<\#\Gamma_v<\infty$, then we associate the relation $g^k=1$ where $k$ is the order of $g$, and we call $g$ an \emph{elliptic cycle}.  Otherwise, if $\#\Gamma_v=\infty$, then we associate the empty relation, a \emph{parabolic cycle}.  We note that the latter occurs if and only if $g$ has infinite order if and only if $\trd(g)=\pm 2$, so the relation $g$ is computable.

We now appeal to the structure theory for Fuchsian groups with cofinite area \cite[\S 4.3]{Katok}.  Suppose that $\Gamma$ has exactly $t$ elliptic cycles of orders $m_1,\dots,m_t \in \Z_{\geq 2}$ and $s$ parabolic cycles, and that $X=\Gamma \backslash \frakH$ has genus $g$.  We say then that $\Gamma$ has signature $(g;m_1,\dots,m_t;s)$.  Moreover, $\Gamma$ is generated by elements
\begin{equation} \label{gens}
\alpha_1,\dots,\alpha_g,\beta_1,\dots,\beta_g,\gamma_1,\dots,\gamma_t,\gamma_{t+1},\dots,\gamma_{t+s}
\end{equation}
subject to the relations
\begin{equation} \label{relats}
\gamma_1^{m_1}=\dots=\gamma_t^{m_t}=[\alpha_1,\beta_1] \cdots [\alpha_g,\beta_g]\gamma_1\cdots \gamma_{t+s} =1,
\end{equation}
where $[\alpha,\beta]=\alpha\beta\alpha^{-1}\beta^{-1}$ is the commutator.  (One obtains a minimal set of generators from this presentation by eliminating $\gamma_{t+s}$ whenever $t+s>0$.)

From the set of generators coming from the side-pairing elements and the set of relations coming from the minimal cycles, we can build a minimal set of generators and relations by ``back substitution''.  First, we prove a lemma.

\begin{lem} \label{freeproduct}
Suppose $\Gamma \cong \Gamma_1 * \Gamma_2$ is a free product, and that $\gamma_i \in \Gamma_1$ or $\Gamma_2$ for $i=1,\dots,s+t$.  Then either $\Gamma_1$ or $\Gamma_2$ is isomorphic to the free product of cyclic groups.
\end{lem}

\begin{proof}
Let $\phi:\Gamma \xrightarrow{\sim} \Gamma_1 * \Gamma_2$ be an isomorphism.  Passing to the quotient by the $\gamma_i$, for $i=1,\dots,t+s$, we may assume that $s=t=0$.  But then the homology groups $H_i(\Gamma,\Z)$ (coming from group homology) coincide with the homology groups $H_i(Y,\Z)$ (coming from topology) where $Y$ is the orientable surface of genus $g$ \cite[\S II.4]{Brown}; in particular, we have $H_0(\Gamma,\Z)=\Z$.  By the Mayer-Vietoris sequence \cite[Corollary II.7.7]{Brown}, we have 
\[ \Z \cong H_0(\Gamma,\Z) \cong H_0(\Gamma_1*\Gamma_2,\Z) \cong H_0(\Gamma_1,\Z) \oplus H_0(\Gamma_2,\Z) \]
so say $H_0(\Gamma_2,\Z)=0$; but this immediately implies $\Gamma_2$ is trivial as well, and the result now follows.
\end{proof}

\begin{alg} \label{minrelations}
Let $P$ be a side pairing for $D$ and let $M$ be a set of minimal cycles for $D$.  This algorithm returns a minimal set of generators and relations for $\Gamma$.
\begin{enumalg}
\item Let $H \subset G(P)$ be such that $g \in G$ implies either $g=g^{-1}$ or $g^{-1} \not \in G$.  
\item Let $R$ be the set of elliptic cycles in $M$ and let $A$ be the set of accidental cycles.  Initialize $r$ to be an element of $A$ and remove $r$ from $A$.
\item If $A = \emptyset$, add $r$ to $R$ and return the generators $H$ and the relations $R$.  Otherwise, choose an element $g \in A$ such that $g$ and $r$ have an element $g_i \in H$ in common; then solve for $g_i$, substitute this expression in for $g_i$ in the relation $r$, and remove $g_i$ from $H$.  Return to Step 3.
\end{enumalg}
\end{alg}

\begin{proof}[Proof of correctness]
If in Step $3$ there is always an element $g \in A$ such that $g$ and $r$ have an element in common, then the algorithm terminates correctly: in the notation of (\ref{gens}--\ref{relats}), there are exactly $t+1$ relations, and hence the set of generators must also be minimal.

So suppose otherwise.  Let $H_1$ be the set of $g \in H$ such that $g$ or $g^{-1}$ occurs in the relation $r$ and let $H_2=H \setminus H_1$.  Let $\Gamma_1,\Gamma_2$ be the groups generated by $H_1,H_2$.  Then by assumption, $\Gamma$ is the free product of $\Gamma_1$ and $\Gamma_2$.  By Lemma \ref{freeproduct}, since the relation in $\Gamma_1$ is nontrivial, it follows that $\Gamma_2$ is the free product of finite cyclic groups, and hence cannot contain any accidental cycles, which is a contradiction.  
\end{proof}

The minimal presentation resulting from Algorithm \ref{minrelations} is not necessarily of the form (\ref{gens})--(\ref{relats}); we refer to the methods of Imbert \cite{Imbert} for an alternative approach using fat graphs which computes such a canonical presentation.

This completes the proof of the theorem and the accompanying corollaries in the introduction.  

\begin{rmk}
If in the first corollary, one wants the structure of $\calO^*$, we use the exact sequence
\[ 1 \to \Z_F^{*2} \calO_1^* \to \calO^* \xrightarrow{\nrd} \Z_{F,+}^*/\Z_F^{*2} \to 1 \]
where $\Z_{F,+}^*=\{u \in \Z_F^* : v(u)>0\text{ for all ramified places $v \mid \infty$}\}$.
From the solution to the word problem, it then suffices to find elements $\gamma \in \calO^*$ such that $\nrd(\gamma)=u$ generates the finite group $\Z_{F,+}^*/\Z_F^{*2}$, and these can be found using the methods of \S 3.
\end{rmk}

\section{Examples} \label{examples}

We have implemented a variant of the above algorithm in the computer system \textsf{Magma} \cite{Magma}.  In this section, we provide two examples of the output of this algorithm.

\begin{figure}[h]
\begin{center}
\includegraphics[width=4.5in]{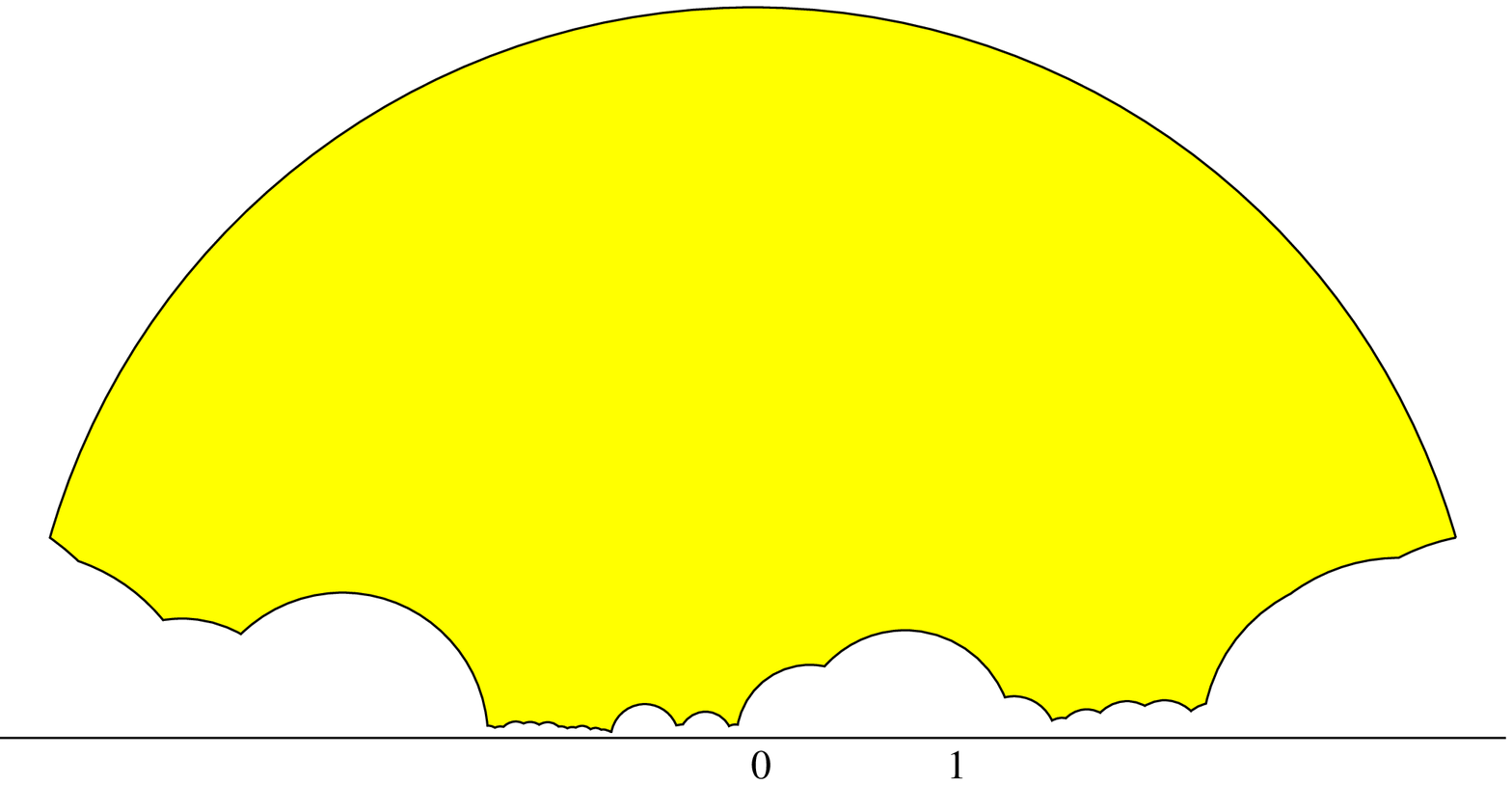} \\
\textbf{Figure \ref{examples}.1}: A Dirichlet domain for the arithmetic Fuchsian group $\Gamma_0^{6}(13)$
\end{center}
\end{figure}

First, we consider the quaternion algebra $B=\quat{3,-1}{\Q}$ of discriminant $6$.  A maximal order $\calO$ is given by
\[ \calO = \Z \oplus \Z\alpha \oplus \Z\beta \oplus \Z\frac{1+\alpha+\beta+\alpha\beta}{2}. \]
We consider the Eichler order contained in $\calO$ of level $13$, given by
\begin{align*} 
\calO(13) &= \Z \oplus \Z\frac{3-5\alpha-5\beta+3\alpha\beta}{2} \oplus \Z(2-2\alpha-\beta+\alpha\beta) \\ & \qquad \oplus \Z\frac{13-13\alpha-13\beta+13\alpha\beta}{2}.
\end{align*}
We denote $\Gamma(\calO)=\Gamma_0^{6}(13)$.  We embed $B \hookrightarrow M_2(\R)$ by the embedding (\ref{embedmin}), and take $p=9i/10 \in \frakH$.  By (\ref{shimizu}), we compute that the Fuchsian group $\Gamma_0^{6}(13)$ has coarea $14/3$.

Step 2 in Algorithm \ref{algDO} finds the units $(1-\alpha-3\beta+\alpha\beta)/2,\alpha-2\beta,\dots$, and following the algorithm, reduction and further enumeration automatically yields the fundamental domain as in Figure \ref{examples}.1.  (The methods in \textsf{Magma} for producing the postscript graphic are due to Helena Verrill \cite{Verrill}.)

This domain already exhibits significant complexity: it has $38$ sides and hence $19$ side-pairing elements, which yields a set of $10$ minimal generators $\gamma_1,\dots,\gamma_{10}$ for $\Gamma_0^{6}(13)$, namely
\begin{center}
$12-7\alpha+4\beta+2\alpha\beta,\ (1-\alpha-33\beta-19\alpha\beta)/2,\ 2\alpha+16\beta+9\alpha\beta$, \\
$(37-19\alpha+9\beta+11\alpha\beta)/2,\ 2\alpha + 4\beta + \alpha\beta,\ (1-\alpha-3\beta+\alpha\beta)/2$, \\
$\alpha-2\beta,\ (1+7\alpha-15\beta-5\alpha\beta)/2,\ (1+7\alpha-45\beta-25\alpha\beta)/2,\ \alpha-14\beta-8\alpha\beta$,
\end{center}
subject to the relations
\begin{eqnarray*}
\gamma_3^2=\gamma_5^2=\gamma_7^2=\gamma_{10}^2=\gamma_2^3=\gamma_6^3=\gamma_8^3=\gamma_9^3=1 \\
\gamma_1^{-1}\gamma_4\gamma_5\gamma_6^{-1}\gamma_1\gamma_2^{-1}\gamma_3\gamma_4^{-1}\gamma_7\gamma_8^{-1}\gamma_9^{-1}\gamma_{10}^{-1}=1.
\end{eqnarray*}
We deduce that $\Gamma_0^{6}(13)$ has signature $(1;2,2,2,2,3,3,3,3;0)$, a fact which can be independently verified by well-known formulae \cite{AB}.

Second, we consider the totally real number field $F$ generated by a root $t$ of the polynomial $x^7-x^6-6x^5+4x^4+10x^3-4x^2-4x+1$; it is the minimal septic totally real field, having discriminant $d_F=20134393=71 \cdot 283583$.  We consider the quaternion algebra $B$ which is ramified at $6$ of the $7$ real places of $F$ and no finite place: explicitly, $B=\quat{h,k}{F}$ where $h=-t^6+6t^4+t^3-9t^2-3t+1$ and $k=-t^2+2t-1$, and in fact $h,k \in \Z_F^*$.  We compute a maximal order $\calO$ of $B$.  Letting $\Gamma=\Gamma(\calO)$, we see that $\Gamma$ has coarea $5/2$.  The output of Algorithm \ref{algDO} in this case is given in Figure \ref{examples}.2; we find that $\Gamma$ has signature $(0;2,2,2,2,2,3,3,3;0)$.

\begin{figure}[ht]
\begin{center}
\includegraphics[width=4.5in]{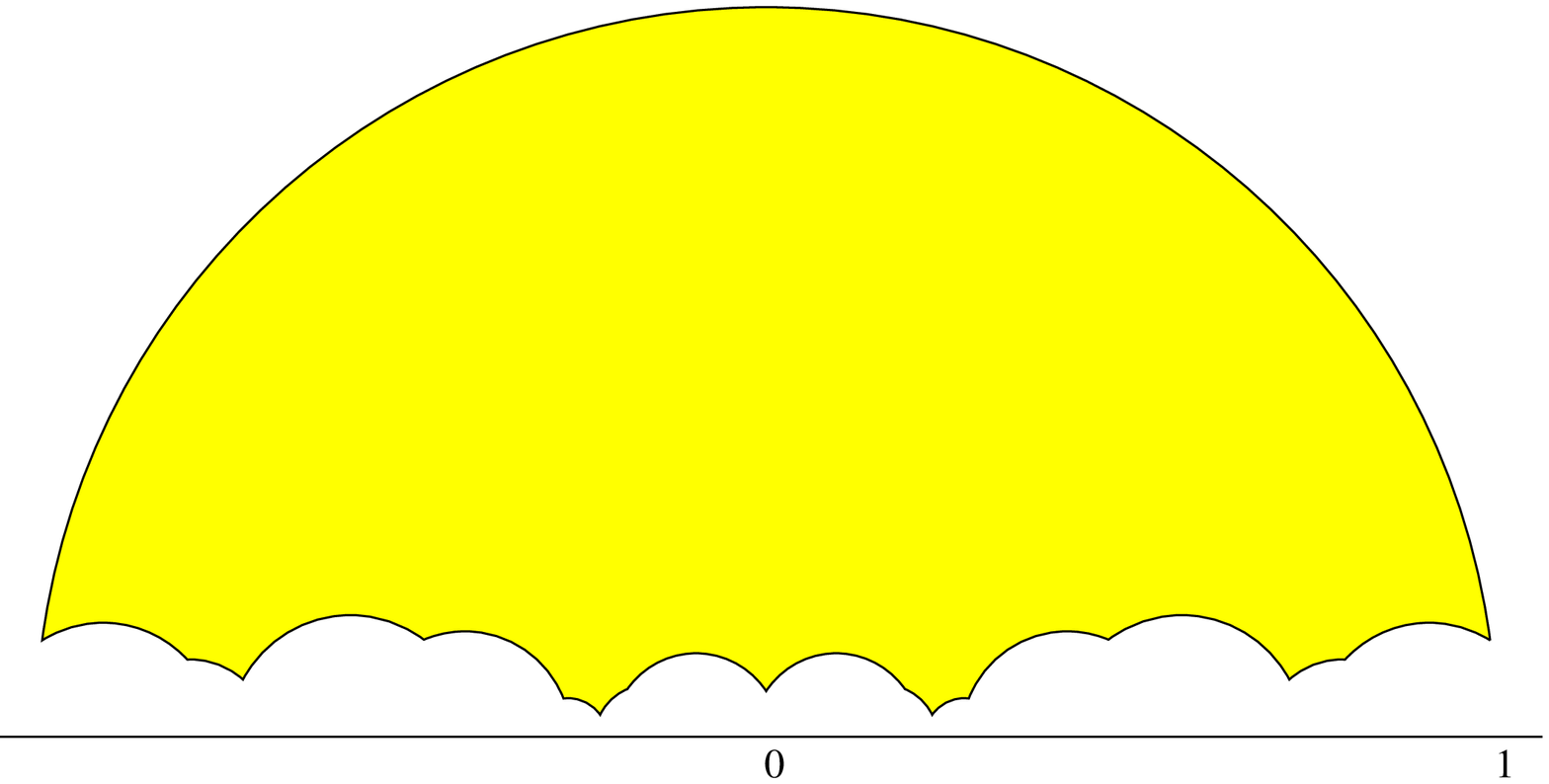} \\
\textbf{Figure \ref{examples}.2}: A Dirichlet domain for the arithmetic Fuchsian group $\Gamma$ \\
associated to a quaternion algebra over the minimal septic totally real field
\end{center}
\end{figure}

We conclude by noting that it would be interesting to extend the methods in this paper to other arithmetic groups; this would allow the computation of unit groups for a wider range of quaternion algebras over number fields and would have further consequences for the algorithmic theory of Shimura varieties.


\end{document}